\documentclass[reqno]{article}
 
\usepackage{ifxetex}
\ifxetex
  \usepackage{fontspec}
\else
  \usepackage[T1]{fontenc}
  \usepackage[utf8]{inputenc}
  \usepackage{lmodern}
\fi

\usepackage{lineno,hyperref}
\usepackage{hyperref,geometry,graphicx,amsmath,amssymb,amsfonts,amsthm,array}
\usepackage{float}
\usepackage{multicol, indentfirst}
\usepackage{amsmath}
\usepackage{amssymb}
\usepackage{comment}
\usepackage{calrsfs}
\usepackage{mathtools}
\usepackage{mathrsfs}
\usepackage{amsfonts}
\usepackage{calligra}
\usepackage{textcomp}
\usepackage{upgreek}

\DeclarePairedDelimiter\floor{\lfloor}{\rfloor}

\DeclareMathAlphabet{\mathcal}{OMS}{cmsy}{m}{n}

\newcounter{counter} \numberwithin{counter}{section}

\newtheorem{theorem}[counter]{Theorem}
\newtheorem{proposition}[counter]{Proposition}
\newtheorem{lemma}[counter]{Lemma}
\newtheorem{corollary}[counter]{Corollary}
\newtheorem{definition}[counter]{Definition}
\newtheorem{example}[counter]{Example}

\newtheorem{construction}[counter]{Construction}

\newcommand\numberthis{\addtocounter{equation}{1}\tag{\theequation}}

\numberwithin{equation}{section}

\newcommand{\adj}{\text{adj}}
\newcommand{\phimap}[1]{\varphi_{#1}}

\newcommand{\R}{\mathbb{R}}
\newcommand{\C}{\mathbb{C}}
\newcommand{\I}{\mathcal{I}\left(S\right)_{n-1}}
\newcommand{\It}{\mathcal{I}\left(\tilde{S}\right)_{n-1}}
\newcommand{\phieig}[1]{\Lambda\left( #1 \right)}
\newcommand{\E}[1]{\Lambda\left( #1 \right)_{n-1}}
\newcommand{\Ek}[1]{\Lambda\left( #1 \right)_{k}}
\newcommand{\F}[1]{\Lambda\left( #1 \right)_{1}}
\newcommand{\Equartic}[1]{\Lambda\left( #1 \right)_{3}}
\newcommand{\Equintic}[1]{\Lambda\left( #1 \right)_{4}}
\newcommand{\El}{\Lambda \left(\omega^{\ell}\right)_{n-1}}

\newcommand{\cpts}{\mathcal{V}_{\mathbb{C}}\big(f,\frac{\partial f}{\partial t}\big)}
\newcommand{\rpts}{\mathcal{V}_{\mathbb{R}}\big(f,\frac{\partial f}{\partial t}\big)}

\title{Determinantal representations of invariant hyperbolic plane curves}
\author{Konstantinos Lentzos and Lillian Pasley}

\begin{document}



\date{}
\maketitle

\bibliographystyle{abbrv}




\begin{abstract}
We study hyperbolic polynomials with nice symmetry and express them as the determinant of a Hermitian matrix with special structure. The goal of this paper is to answer a question posed by Chien and Nakazato in 2015. By properly modifying a determinantal representation construction of Dixon (1902), we show for every hyperbolic polynomial of degree $n$ invariant under the cyclic group of order $n$ there exists a determinantal representation admitted via some cyclic weighted shift matrix. Moreover, if the polynomial is invariant under the action of the dihedral group of order $n$, the associated cyclic weighted shift matrix is unitarily equivalent to one with real entries.
\end{abstract}



\section{Introduction} \label{intro}

    Let $f$ be a real homogeneous polynomial of degree $n$ in three variables $t,x,y$, so $\mathcal{V}_{\mathbb{C}}(f)$ is a projective plane curve. A determinantal representation of $f$ is an expression
\[ f = \det(t M_0 + xM_1 + y M_2), \]
where $M_0, M_1, M_2$ are $n\times n$ matrices.
We set $M=M(t,x,y) = tM_0+xM_1+yM_2$ and refer to $M$ as the determinantal representation of $f$.
The representation is called real symmetric or Hermitian if $M$ is of the respective form.
Real symmetric and Hermitian determinantal representations have been systematically studied by Dubrovin \cite{Dub83} and Vinnikov \cite{Vinnikov89, Vinnikov93} in the late 1980's and early 1990's. 
Definite Hermitian determinantal representations are those for which there exists a point $e=(e_0,e_1,e_2) \in \mathbb{R}^3$ such that the 
matrix $M(e) = e_0 M_0 + e_1 M_1 + M_2 e_2$ is positive definite. Since the eigenvalues of a Hermitian matrix are real, every real line passing
through $e$ meets the hypersurface $\mathcal{V}_{\mathbb{C}}(f)$ in only real points. Polynomials with this property are called hyperbolic
and are intimately linked with convex optimization, see for example \cite{Bau98}, \cite{Guler97} and \cite{Renegar2006}.

\begin{definition}
 A homogeneous polynomial $f\in \mathbb{R}[t,x,y]_n$ is called \emph{hyperbolic with respect to a point $e\in\mathbb{R}^3$} if $f(e)\neq 0$ 
 and for every $z\in\mathbb{R}^3$, all roots of the univariate polynomial $f(e+\lambda z)\in\mathbb{R}[\lambda]$ are real.
\end{definition}

Hyperbolicity is reflected in the topology of $\mathcal{V}_{\mathbb{R}}(f)$ in the real projective plane $\mathbb{P}^2(\R)$. In particular, if $\mathcal{V}_{\mathbb{C}}(f)$ is smooth, 
$f$ is hyperbolic if and only if $\mathcal{V}_{\mathbb{R}}(f)$ consists of   $\floor{\frac{n}{2}}$ nested ovals, as well as a pseudo-line if $n$ is odd.
Lax conjectured, in the context of hyperbolic differential operators, that every hyperbolic polynomial possesses a definite determinantal
representation with real symmetric matrices  \cite{Lax58}. Helton and Vinnikov \cite{HV2007} proved this conjecture in 2007, while Plaumann and Vinzant \cite{PlaumannVinzant} gave a concrete construction
for the Hermitian case in 2013.

    \begin{definition}
          The {\rm{numerical range}} of $A \in \C^{n \times n}$ is $\mathcal{W}(A) := \{x^{\ast} A \hspace{.5mm} x\in \C \mid x \in \C^n, \hspace{1mm} x^{\ast}x = 1\}\text{.}$ 
    \end{definition}
    
 The numerical range is compact and convex in $\C$, or equivalently, $\R^2$ \cite{haus,toep} and invariant under unitary transformation \cite{kippenhahn}. That is, for any $B = UAU^{\ast}$ where $UU^{\ast} = I$, the equality $\mathcal{W}(A) = \mathcal{W}(B)$ holds. Geometrically, this set is an affine projection of the semidefinite cone \cite{henrion}.

    \begin{definition}
        For any matrix $A \in \C^{n \times n}$, let $$\label{rep} f_{A}(t,x,y) := \det(tI_n + x \Re(A) + y \Im(A))$$
where $\Re(A)=\tfrac{A+A^{\ast}}{2}$ and $\Im(A)=\tfrac{A-A^{\ast}}{2i}$. The dual of the algebraic curve in $\mathbb{P}^2(\C)$ defined by the zero set of $f_A$ is called the {\rm{boundary generating curve}} of $\mathcal{W}(A)$.
    \end{definition}

            \begin{figure}\vspace{-25mm}
            \includegraphics[scale=0.35]{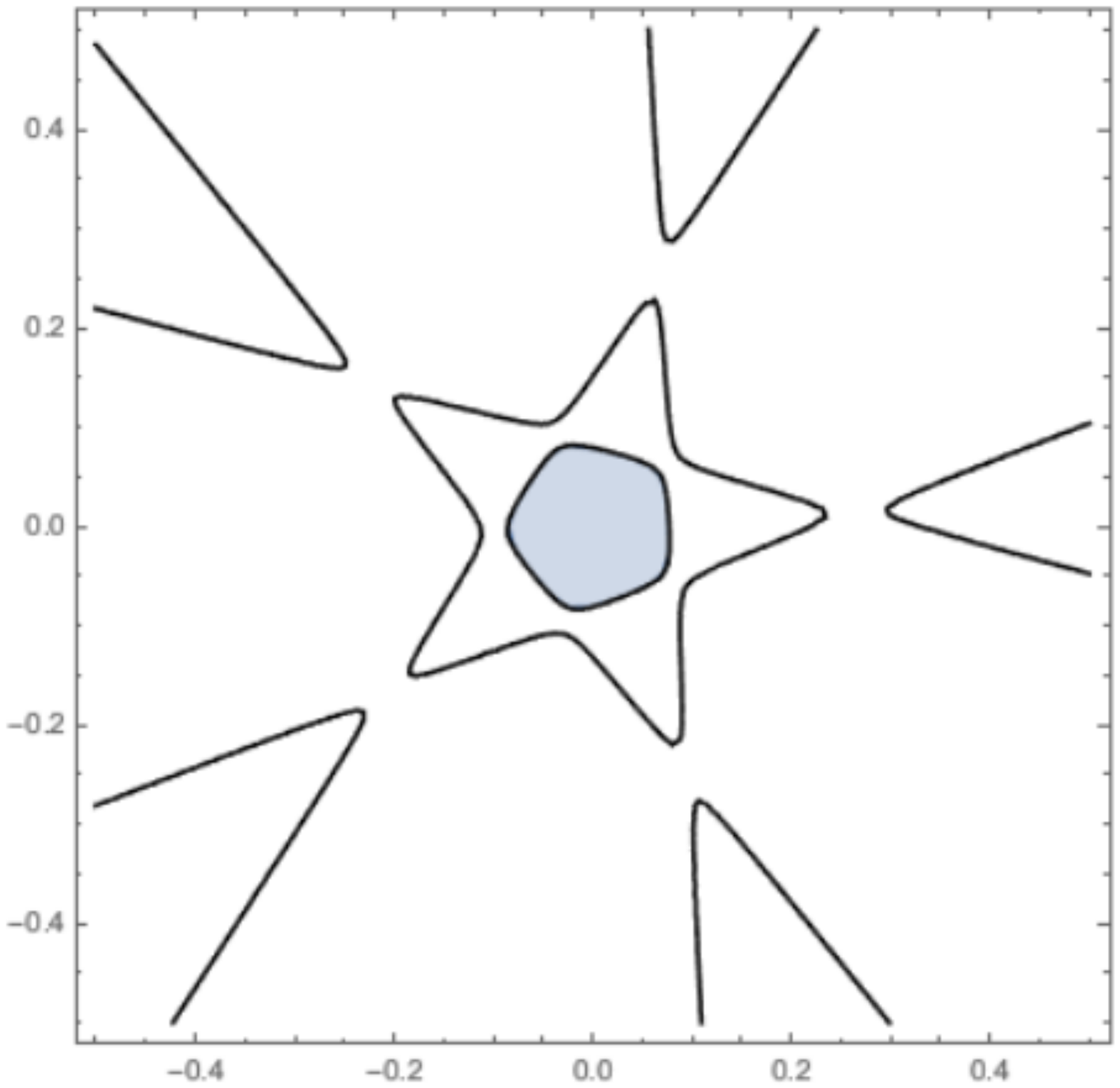}\centering
            \includegraphics[scale=0.35]{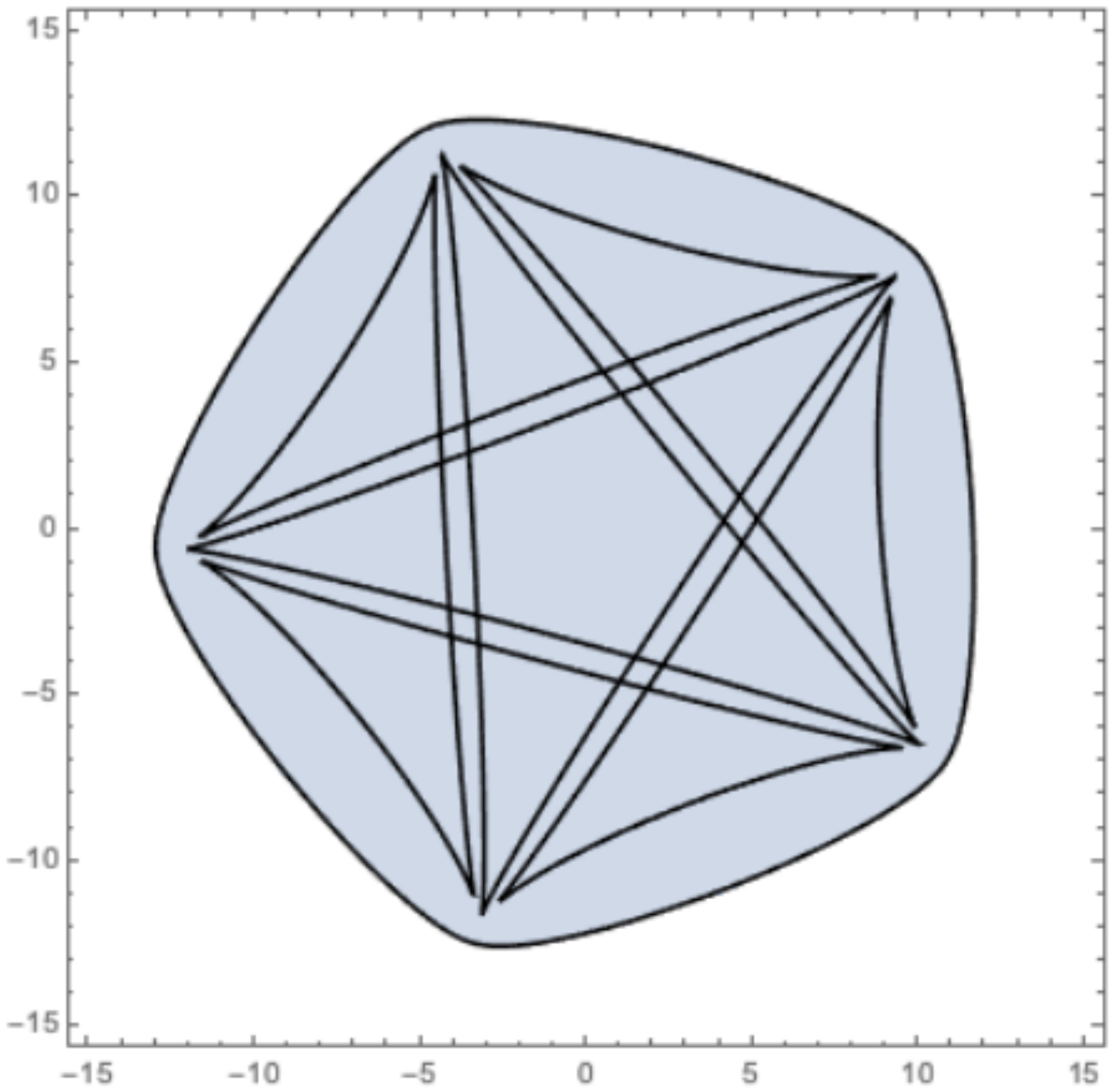}\centering 
            \vspace{-25mm}
            \caption{Hyperbolic dihedral invariant quintic $f_A$ associated to $A = S(10+5i, 2+10i, 14, 16i, 12)$ in the hyperplane $\{t=1\}$ with connected component containing origin shaded (left). The dual curve of $f_A$ with shaded convex hull corresponding to $\mathcal{W}(A)$ (right).}
        \end{figure}


The convex hull of the boundary generating curve is exactly $\mathcal{W}(A)$ \cite{kippenhahn}.  It is shown in \cite{Chien2013} that $f_A$ associated with cyclic weighted shift matrix $A \in \mathbb{R}^{n \times n}$ is hyperbolic with respect to $(1,0,0)$ and is invariant under the action of the dihedral group $D_n$.

\begin{definition}
 A cyclic weighted shift matrix $A=S(a_1,\dots, a_n) \in \mathbb{C}^{n \times n}$ has the form 
   \[A = S(a_1,\dots,a_n) = \left( \begin{array}{cccccc}
0      & a_1     & 0      & 0      & \cdots  & 0       \\
0      & 0       & a_2    & 0      & \cdots  & 0       \\
\vdots & \vdots  & \ddots & \ddots & \ddots  & \vdots  \\
\vdots & \vdots  & \vdots & \ddots & \ddots  & 0       \\
0      & 0       & \vdots & \vdots & \ddots  & a_{n-1} \\
a_{n}  & 0       & \cdots & \cdots & \cdots  & 0       \\
\end{array} \right)\text{.}\]
\end{definition}

Chien and Nakazato posed the converse problem and were interested in its connection to numerical ranges \cite{Chien2015}. In view of \cite{PlaumannVinzant}, we prove that a hyperbolic plane curve, invariant under the action of the cyclic group $C_n$, has a determinantal representation admitted by some cyclic weighted shift matrix with complex entries. Furthermore, we show if the plane curve is invariant under the action of the dihedral group $D_n$ that we can recover real entries in the associated cyclic weighted shift matrix. Now we discuss the action of the cyclic and dihedral groups, introduce our main theorem, and establish a connection between $f_A$ and $\mathcal{W}(A)$.

    Let $f \in\mathbb{R}[t,x,y]$ and define
$\Phi = \left(\hspace{1mm} \begin{smallmatrix}
                1      & 0                    & 0             \\
                0      & \cos\left(\frac{2 \pi}{n}\right)          & -\sin\left(\frac{2 \pi}{n}\right)  \vspace{1mm} \\
                0      & \sin\left(\frac{2 \pi}{n}\right)           & \cos\left(\frac{2 \pi}{n}\right)
\end{smallmatrix} \right)$ and  $\Gamma = \left( \begin{matrix}
1      & 0                    & 0             \\
0      & 1           & 0   \\
0      & 0           & -1  
\end{matrix} \right)$.
The cyclic group $C_n$ and dihedral group $D_n$ can be expressed as
            \begin{align}\label{rep1}
                C_{n} = \langle \Phi \rangle\text{ and }
                D_{n} = \langle \Phi, \Gamma \rangle 
            \end{align} 
where $\Phi$ is the rotation around the axis of $t$ by the angle $\frac{2\pi}{n}$ and $\Gamma$ is the reflection over the $tx$-plane. These groups describe the rotations --- and, in the dihedral case, reflections --- of a regular polygon with $n$ sides. If $f(\Xi(t,x,y))~=~f(t,x,y)$ for every generator $\Xi$ of group $G$, then we write
    \begin{equation} 
        f~\in~\R[t,x,y]^G
    \end{equation}
    and say $f$ is \textit{invariant under the action of $G$}. We provide an alternative proof for a result of Chien and Nakazato \cite{Chien2013}.
         

        \begin{proposition} \label{cyc}
    Let $A=S(a_1,\dots,a_n) \in \mathbb{C}^{n \times n}$ be a cyclic weighted shift matrix. Then $f_A$ is hyperbolic with respect to $e=(1,0,0)$ and invariant under the action of the cyclic group of order $n$. Moreover, if $A \in \R^{n \times n}$, then $f_A$ is invariant under the action of the dihedral group of order $n$.
        \end{proposition}

\begin{proof}
    Notice $f_A(e)=\det({I_n})=1\neq 0$ and $$f_A(e+\lambda z)=\det\left(I+\lambda\left(z_1I + \frac{1}{2}(z_2-iz_3)A + \frac{1}{2}(z_2+iz_3)A^{\ast}\right)\right)$$ is the characteristic polynomial of a Hermitian matrix, which has all real eigenvalues. Therefore, $f_A$ is hyperbolic with respect to the point $e=(1,0,0)$. Write $$f_A(t,x,y) = \det\left(tI_n+\left(\frac{x+iy}{2}\right)A^{\ast}+\left(\frac{x-iy}{2}\right)A\right)$$ and let $\Omega = \text{diag}(1,\omega,\dots,\omega^{n-1})$ for $\omega=e^{\frac{2\pi i}{n}}$. Then
    \begin{align*}
    f_A\big(\Phi(t,x, y)\big) &= \det\left(tI_n + \omega \left(\frac{x+iy}{2}\right) A^{\ast} + \omega^{n-1}\left(\frac{x-iy}{2}\right)A\right) \\
    &=\det\left(\Omega\left(tI_n + \frac{x+iy}{2} A^{\ast} + \frac{x-iy}{2}A\right)\Omega^{\ast}\right) \\
    &=\det(\Omega)f_A(t,x,y)\det(\Omega^{\ast}) \\
    &= f_A(t,x,y)
    \end{align*}
    so $f_A$ is invariant under the action of rotation and $f_A \in \R[t,x,y]_n^{C_n}$.
    Now assume $A \in \R^{n \times n}$, so $A^{\ast} = A^T$. Then
  \begin{align*}
  f_A\left(\Gamma(t,x,y)\right) &= \det\left(tI_n+\frac{x-iy}{2}A^{T} + \frac{x+iy}{2} A\right)  \\
    &= \det\left(\left(tI_n+\frac{x+iy}{2}A^T + \frac{x-iy}{2}A\right)^T\right) =f_A(t,x,y)
    \end{align*}
    so $f_A$ is also invariant under the action of reflection and $f_A \in \R[t,x,y]_n^{D_n}$.
\end{proof}

Chien and Nakazato were naturally interested in asking the inverse question. Given a curve with dihedral (or cyclic) invariance, can we always find an associated cyclic weighted shift matrix? If so, is there one with all real entries? The main theorem of our paper gives this question a positive answer.

\begin{theorem}\label{main}
Let $f \in\mathbb{R}[t,x,y]_{n}$ be hyperbolic with repect to $(1,0,0)$ with $f(1,0,0)=1$. 
\begin{itemize}
    \item[a.] If $f \in \R[t,x,y]^{C_n}_n$, then
there exists cyclic weighted shift matrix $A\in\C^{n \times n}$ such that
            $f = f_A$. 
        \item[b.] If $f \in \R[t,x,y]^{D_n}_n$, then there exists cyclic weighted shift matrix $B \in \R^{n \times n}$ such that $f = f_B$.
    \end{itemize}
\end{theorem}

The example in Section 3 of \cite{Chien2015} shows there exists matrix $A$ so $f_A \in \C[t,x,y]^{C_n}_n$ and $A$ is not unitarily equivalent to any cyclic weighted shift matrix (with positive weights). Theorem \ref{main} proves there must exist some cyclic weighted shift matrix with the same numerical range of $A$, even if they are not unitarily equivalent.
\begin{corollary}
    If $f \in \C[t,x,y]^{C_n}_n$, then for any matrix $A$ such that $f = f_A$, there exists cyclic weighted shift matrix $B$ with $\mathcal{W}(A) = \mathcal{W}(B)$.
\end{corollary}

    \begin{proof}
        By Theorem \ref{main}, there exists cyclic weighted shift matrix $B$ so that $f_A=f_B$. Then by \cite{kippenhahn}, the boundary generating curves of $\mathcal{W}(A)$ and $\mathcal{W}(B)$ are the same and the equality holds.
    \end{proof}

The rest of the paper is organized as follows. We present a useful change of variables in Section \ref{cov} and calculate the number of polynomial invariants that generate $\C[t,x,y]_n^{C_n}$. In Section~\ref{prereq}, we establish helpful facts about the polynomial $f$ under specified group actions. We give a proof of Theorem \ref{main}(a) for $\mathcal{V}_{\C}(f)$ smooth based on a construction due to Dixon \cite{Dixon} in Section 3. We then extend this result to any curve with cyclic invariance in Section \ref{singular} as we consider curves with $\mathcal{V}_{\C}(f)$ singular. In Section \ref{dihedral}, we prove Theorem \ref{main}(b) and summarize our work and discuss further generalizations in the last section.

\section{Polynomial Invariants and a Change of Variables} \label{cov}

Here we introduce a change of variables and discuss the resulting actions of $\Phi$ and $\Gamma$ under this map. Then we precisely define elements of $\C[t,u,v]_n^{C_n}$.  First let 
    \begin{equation}
\text{conj} : [t:x:y] \mapsto [\overline{t}:\overline{x}:\overline{y}]
    \end{equation} denote the action of conjugation. Consider the change of variables given by the map
    \begin{align*}
       \xi: \mathbb{P}^2(\mathbb{C}) &\to \mathbb{P}^2(\mathbb{C})\text{, }[t:x:y] \mapsto [t:x+iy:x-iy]=[t:u:v]  \numberthis \\
       \xi^{-1}: \mathbb{P}^2(\mathbb{C}) &\to \mathbb{P}^2(\mathbb{C})\text{, }[t:u:v] \mapsto \Big[t:\frac{u+v}{2}:\frac{u-v}{2i}\Big]\text{.}
    \end{align*}
Notice $u \neq \overline{v}$ when $x,y \in \C\backslash\R$, and the action of conjugation is
\begin{equation}
    \text{conj} : [t:u:v] \mapsto \big[\overline{t}:\overline{v}:\overline{u}\big]\text{.}
\end{equation} In terms of group actions of $C_n$ and $D_n$ on $\C[t,u,v]$, our convention is to first apply the actions of $\Phi$ or $\Gamma$ to points $[t:x:y]$ and \textit{then} the change of variables $\xi$. Consequently, the compositions are given by
    \begin{align*}
        \xi\circ\Phi^{\ell}: [t:x:y] &\mapsto \big[t:\omega^{\ell} u: {\omega^{-\ell}}v\big]  \numberthis\\
         \xi\circ\Gamma: [t:x:y] &\mapsto [t:v:u]
    \end{align*}
        for $\omega  = e^{\frac{2\pi i}{n}}$ and some $\ell \geq 0$. These actions give an equivalent representation to (\ref{rep1}), so
        \begin{equation}
        C_{n} = \left\langle \tilde{\Phi} \right\rangle\text{ and }         D_{n} = \left\langle \tilde{\Phi}, \tilde{\Gamma} \right\rangle
        \end{equation}

\noindent where $\tilde{\Phi} = \left(\hspace{1mm} \begin{matrix}
                1      & 0                    & 0             \\
                0      & \omega          & 0  \vspace{1mm} \\
                0      & 0          & \omega^{-1}
\end{matrix} \right)$ for $\omega=e^{\frac{2\pi i}{n}}$ and  $\tilde{\Gamma} = \left( \begin{matrix}
1      & 0                    & 0             \\
0      & 0           & 1   \\
0      & 1           & 0  
\end{matrix} \right)$ act on points $[t:u:v]$. Under this change of variables, the form (\ref{rep}) becomes
             \begin{equation}
                 f_A(t,u,v) = \det\left(tI_n + \left(\frac{u+v}{2}\right)\Re(A) + \left(\frac{u-v}{2i}\right) \Im(A)\right) = \det\left(tI_n + \frac{u}{2} A^{\ast} + \frac{v}{2} A\right)\text{.}
             \end{equation}
For our purposes, we define 
             \begin{equation}
                 \R[t,u,v]:= \C[t,u,v]^{\text{\vspace{1mm}conj}}
             \end{equation}
and the hyperbolicity condition that $f(t,\cos(\theta), \sin(\theta))$ has all real roots for all $\theta \in [0, 2\pi)$ is equivalent to 
     \begin{equation} \label{hyp}
     f(\xi(t,\cos(\theta), \sin(\theta))) = f(t,e^{i\theta},e^{-i\theta})
     \end{equation}
     having all real roots for all $\theta \in [0, 2\pi)$ where $f \in \R[t,u,v]$. Now we only consider polynomials invariant under the action of $C_n$ and later examine the more specific dihedral case in Section \ref{dihedral}. 
        \begin{proposition}\label{dimCnn}
            The degree $n$ part of the invariant ring $\C[t,u,v]^{C_n}$ has dimension $\left \lfloor \frac{n}{2} \right\rfloor + 3$.
        \end{proposition}

        \begin{proof}
            Let the Hilbert series $H(\C[t,u,v]^{C_n},z) = \sum \limits_{k=0}^{\infty} \alpha_k z^k$ where $\dim\big(\C[t,u,v]^{C_n}_k\big) = \alpha_k$ for every~$k$. By Theorem 2.2.1 of \cite{sturmfels}, the Hilbert series is given by
            \begin{align*}
                H(\C[t,u,v]^{C_n},z) &= \frac{1}{|C_n|} \sum \limits_{\Xi \in C_n} \frac{1}{\det(I - z\hspace{1mm}\Xi)} \\&= \frac{1}{n} \sum \limits_{\ell=0}^{n-1} \frac{1}{\det\big(I - z\hspace{1mm}\tilde{\Phi}^{\ell}\big)} = \frac{1}{n} \left(\frac{1}{1-z}\right) \sum \limits_{\ell=0}^{n-1} \frac{1}{(1-\omega^{\ell}z)(1-\omega^{-\ell}z)}\text{.}
            \end{align*}
      Expand the inner term, so
      $$
                \frac{1}{(1-\omega^{\ell}z)(1-\omega^{-\ell}z)} = \sum \limits_{i=0}^{\infty} (\omega^{\ell}z)^i\sum \limits_{j=0}^{\infty} (\omega^{-\ell}z)^j 
                = \sum \limits_{i=0}^{\infty} \sum \limits_{j=0}^{i} \omega^{\ell(i-2j)}z^i
            $$
            and the Hilbert series becomes
            \begin{align*}
                H(\C[t,u,v]^{C_n},z) &= \frac{1}{n} \left(1+z+z^2+\cdots\right) \sum \limits_{\ell=0}^{n-1} \left(\sum \limits_{i=0}^{\infty} \sum \limits_{j=0}^{i} \omega^{\ell(i-2j)}z^i \right)\text{.}
            \end{align*}
            In this expansion, we want to calculate the coefficient $\alpha_n$. More explicitly,
            \begin{align*}
                \alpha_n &= \frac{1}{n} \sum \limits_{\ell=0}^{n-1} \left(\sum \limits_{i=0}^{n} \sum \limits_{j=0}^{i} \omega^{\ell(i-2j)}\right) \\
                &=\frac{1}{n} \sum \limits_{\ell=0}^{n-1} \left( (1) + \left(\omega^{\ell} + \omega^{-\ell}\right) + \left(\omega^{2\ell} + 1 + \omega^{-2\ell}\right) + \cdots + \left(\omega^{n \ell} + \omega^{(n-2)\ell} + \cdots + \omega^{-n\ell}\right) \right) \\
                &=\frac{1}{n} \sum \limits_{\ell=0}^{n-1} 1 + \sum \limits_{\ell=0}^{n-1} \left(\omega^{\ell} + \omega^{-\ell}\right) + \sum \limits_{\ell=0}^{n-1}\left(\omega^{2\ell} + 1 + \omega^{-2\ell}\right) + \cdots + \sum \limits_{\ell=0}^{n-1}\left(1 + \omega^{(n-2)\ell} + \cdots + 1\right) \\
                &=\frac{1}{n}\left( n + 0 + n + \cdots +\left(n+0+\cdots+0+n\right)\right) \\
                &= \left\lfloor\frac{n}{2}\right\rfloor + 3\text{.}
                \end{align*}
        \end{proof}
\noindent Let 
    \begin{equation}\beta_1(t,u,v) = t\text{, } \beta_2(t,u,v) = uv\text{, } \beta_3(t,u,v)=\frac{u^n+v^n}{2}\text{, } \beta_4(t,u,v)=\frac{u^n-v^n}{2i}
    \end{equation}where $\beta_i \in \C[t,u,v]^{C_n}$. Then $\dim\left(\C\left[\beta_1,\beta_2,\beta_3,\beta_4\right]_n\right) = \left \lfloor \frac{n}{2} \right \rfloor +3$ and $\C\left[\beta_1,\beta_2,\beta_3,\beta_4\right]_n = \C[t,u,v]^{C_n}_n$ by Proposition \ref{dimCnn}. Therefore, all polynomial invariants of $C_n$ with degree $n$ are generated by $\beta_1, \beta_2, \beta_3$, and $\beta_4$.
In general, any $f \in \R[t,u,v]^{C_n}_n$ can be written
        \begin{equation}\label{fInvz}
             f(t,u,v) = t^n + \sum_{r=1}^{\floor*{\frac{n}{2}}} c_r t^{n-2r} (uv)^{r} +c_0 \left(\frac{u^n+v^n}{2}\right)+\tilde{c_0} \left(\frac{u^n-v^n}{2i}\right) 
        \end{equation}
        for some coefficients $c_i, \tilde{c_0} \in \R$.

\section{Prerequisites} \label{prereq}

In this section, we describe several properties of $f$ and $\frac{\partial f}{\partial t}$ using the form (\ref{fInvz}). We later use these facts in Sections \ref{smooth} and \ref{singular} to prove our main theorem.  The first lemma states that the partial derivative $\frac{\partial f}{\partial t}$ is a product of circles.  In the following proof we consider $f \in \R[t,u,v]$ and use the equivalence from (\ref{hyp}).

        \begin{lemma} \label{circles}
           The partial derivative can be written $\frac{\partial f}{\partial t} = t^{k} q_1 q_2 \cdots q_{\lfloor\frac{n-1}{2}\rfloor}$ where $k =  \Bigg\{
\begin{array}{ll}
       \vspace{2mm} \hspace{-1mm}0, & \text{\hspace{-2mm} $n$ odd} \\
      \hspace{-.5mm}1, & \text{\hspace{-2mm} $n$ even} \\
\end{array}$ for $q_{j} = t^2 - s_j uv$ and $s_j \in \R_{\geq 0}$.
        \end{lemma}

        \begin{proof}
            First write $\frac{\partial f}{\partial t} = nt^{n-1}+\sum \limits_{r=1}^{\lfloor \frac{n}{2} \rfloor} (n-2r)c_r t^{n-2r-1}(uv)^r$. Assume $n$ is odd. Then $n-1=2m$ for some integer $m$, so $$\frac{\partial f}{\partial t} = n{\left(t^{2}\right)^m}+\sum \limits_{r=1}^{\lfloor \frac{n}{2} \rfloor} (n-2r)c_r \left(t^2\right)^{m-r}(uv)^r \in \R[t^2,uv]$$ and factor over $\C$ as $\frac{\partial f}{\partial t}= q_1 q_2 \cdots q_{\frac{n-1}{2}}$ where $q_j = t^2-s_juv$ for some $s_j \in \C$. Since $f$ is hyperbolic with respect to $(1,0,0)$, this means $\frac{\partial f}{\partial t}$ is hyperbolic with respect to $(1,0,0)$ and $\frac{\partial f}{\partial t}(t,e^{i\theta},e^{-i\theta})$ has all real roots for each $\theta \in [0,2\pi)$. This means
            $q_j(t,e^{i\theta},e^{-i\theta}) = t^2-s_j$ has two real roots for every $j$, so $s_j \in \R_{\geq 0}$ for every $j$. If $n$ is even, we can factor out $t$ and proceed with the remaining polynomial of odd degree as before.
        \end{proof}

        The next lemma shows generic $f$ and $\frac{\partial f}{\partial t}$ cannot intersect at the line at infinity when $n$ is odd. Specifically, we require at least one of $c_0$, $\tilde{c_0}$ to be nonzero. The case where $c_0=\tilde{c_0}=0$ occurs when $\mathcal{V}_{\C}(f)$ has singularities and is considered in Section \ref{singular}. With this condition, we establish an explicit description for the points $\cpts$.

        \begin{lemma} \label{intpts}
            If $n$ is odd $f \in \R[t,u,v]^{C_n}_n$ is hyperbolic with at least one of $c_0$, $\tilde{c_0}$ nonzero, then all points in $\cpts$ have $t \neq 0$.
        \end{lemma}

        \begin{proof}
    When $n$ is odd, $\deg\big(\frac{\partial f}{\partial t}\big)=n-1$ is even. By Lemma \ref{circles}, we can write $\frac{\partial f}{\partial t} = q_1 q_2 \cdots q_{\frac{n-1}{2}}$ where $q_j = t^2 - s_juv$ for some $s_j \in \R_{\geq 0}$. If $\frac{\partial f}{\partial t}$ vanishes when $t=0$, then either $u=0$ or $v=0$ as well. Suppose $\frac{\partial f}{\partial t}\left(0,1,0\right)=0$ or $\frac{\partial f}{\partial t}\left(0,0,1\right)=0$. Then, $f(0,1,0)=\frac{c_0-i\tilde{c_0}}{2}$ or $f(0,0,1)=\frac{c_0+i\tilde{c_0}}{2}$, which are both nonzero under assumption. Therefore, $f$ does not vanish in either case.
        \end{proof}
        
Notice $f$, $\frac{\partial f}{\partial t} \in \R[t,x,y]^{C_n}$, so if $[t:x:y] \in \cpts$, then for every $\ell = 0,1, \ldots, n-1$, $\tilde{\Phi}^{\ell}\left([t:u:v]\right)$, $\tilde{\Phi}^{\ell}\left(\text{conj}\left([t:u:v]\right)\right) \in \xi\left(\cpts\right)$. Now, if $\rpts$ is empty, and in particular, $f$ has no real singularities, then these complex intersection points are distinct.
        \begin{proposition} \label{distinct}
        If $\rpts$ is empty and at least one of $c_0$, $\tilde{c_0}$ is nonzero, then $\cpts$ consists of $n(n-1)$ distinct points.
        \end{proposition}

        \begin{proof}
        Suppose $f$ and $\frac{\partial f}{\partial t}$ have a common factor. By Lemma \ref{circles} and since $f$ has no factor of $t$ by assumption, the common factor must be $t^2-s_juv$ for some $s_j \in \R_{\geq 0}$. Then $$\xi^{-1}\left(\left[\sqrt{s_j}:1:1\right]\right) = \left[\sqrt{s_j}:1:0\right] \in \mathcal{V}_{\R}\left(f,\frac{\partial f}{\partial t}\right)\text{,}$$ which is a contradiction. Therefore, $f$ and $\frac{\partial f}{\partial t}$ have no common factors and by Bezout's theorem, $\left|\cpts\right| = n(n-1)$. For distinctness, we need to show for any point in $\cpts$, each orbit under the action of conjugation and rotation is distinct. Suppose for some $\ell \in \{1,\ldots,n-1\}$ that $[1:u:v]=[1:\omega^{\ell}u:\omega^{-\ell}v] \in \xi\left(\cpts\right)$. This implies $\omega^{\ell}u=u$ and $\omega^{-\ell}v=v$ for $\ell \neq 0$, so $u=v=0$ and $[1:0:0] \in \rpts$, which is a contradiction. Now suppose $[1:u:v]=[1:\omega^{\ell}\overline{v}:\omega^{-\ell}\overline{u}] \in \xi\left(\cpts\right)$ for some $\ell \in \{0,\ldots,n-1\}$. 
        These equivalences imply $\omega^{2\ell}u=u$ and $\omega^{-2\ell}v=v$, so either $u=v=0$, $\ell = 0$ or $\ell=\frac{n}{2}$. If $u=v=0$, this is a contradiction as in the previous case. If $\ell=0$, then $[1:x:y]=\xi^{-1}\left([1:u:v]\right) = \xi^{-1}\left([1:\overline{v}:\overline{u}]\right) = [1:\overline{x}:\overline{y}]$, so $x, y \in \R$ and $[1:x:y] \in \rpts$, which is a contradiction. The case $\ell=\frac{n}{2}$ can only happen when $n$ is even since $\ell \in \mathbb{Z}$.
    Assume $n$ be even. If $\ell = \frac{n}{2}$, then $u = -\overline{v}$ and $u\overline{u} = v\overline{v}$. By Proposition \ref{circles}, we can write $\frac{\partial f}{\partial t} = tq_1q_2 \cdots q_{\frac{n-2}{2}}$ for $q_j = t^2-s_juv$. For some $j$ this means 
    $0 = q_j(1,-\overline{v},v)= 1 + s_j v \overline{v}\text{,}$
    which is a contradiction.
    Lastly, suppose $[0:u:v]=[0:\omega^{\ell}u:\omega^{-\ell}v]$ for some $\ell \in \{1,\ldots,n-1\}$. This gives $\omega^{\ell}u=u$ and  \vspace{.5mm} $\omega^{-\ell}v=v$, so $u=v=0$ and this is a contradiction as before.
    \\Now suppose $p \in \xi\left(\cpts\right)$ with multiplicity $m_p \geq 2$. Then $\tilde\Phi^{\ell}(p), \tilde\Phi^{\ell}(\overline{p}) \in \xi\left(\cpts\right)$ each with multiplicity $m_p$. By Lemma \ref{circles}, $q_j(p)=0$ for some factor $q_j$ of $\frac{\partial f}{\partial t}$. Then $$\left|\mathcal{V}_{\C}(f, q_j)\right| = 2n < 2n\cdot m_p\text{,}$$ which is a contradiction. Therefore, each point in $\cpts$ is distinct.
        \end{proof}

         Lemma \ref{intpts} and Lemma \ref{distinct} give an explicit form for the points of intersection.
\noindent The common points have the form $\xi\left(\cpts\right) = S \hspace{.5mm}\cup\hspace{.5mm} \overline{S}$ where
     \begin{equation}\label{S}
            S =
          \begin{cases}{}
        \left\{ [1: \omega^{\ell}u_i: \omega^{-\ell}v_i] \mid 1 \leq i \leq \tfrac{n-1}{2}, 0 \leq \ell \leq n-1\right\}\text{,} & \text{ $n$ odd} \vspace{2mm} \\
        \left\{ [1: \omega^{\ell}u_i: \omega^{-\ell}v_i], [0: \omega^{k} u: \omega^{-k}v]   \mid 1\leq i \leq \tfrac{n-2}{2}, 0 \leq \ell \leq n-1, 0\leq k \leq \tfrac{n-2}{2}\right\}\text{,} & \text{ $n$ even}
           \end{cases}\vspace{1mm}
   \end{equation}
and the point $[1: \omega^{\ell}u_i: \omega^{-\ell}v_i] \in S$ has corresponding point $[1:\omega^{\ell}\overline{v_i},\omega^{-\ell}\overline{u_i}] \in \overline{S}$. Let the set 
        \begin{equation} \label{Stilde}
            \tilde{S} =
            \begin{cases}{}
        \left\{ [1: u_i: v_i] \mid 1 \leq i \leq \tfrac{n-1}{2}\right\}\text{,} & \text{ $n$ odd} \vspace{2mm} \\
        \left\{ [1:u_i:v_i], [0:  u: v]   \mid 1\leq i \leq \tfrac{n-2}{2}\right\}\text{,} & \text{ $n$ even}
           \end{cases}
       \end{equation}
denote a set of orbit representatives so applying the rotation $\tilde{\Phi}^{\ell}$ for $\ell = 0, \ldots, n-1$ to all points in $\tilde{S}$ gives all of $S$.

        \begin{corollary}\label{nosing}
            If $\rpts$ is empty and at least one of $c_0$, $\tilde{c_0}$ is nonzero, then $\mathcal{V}_{\C}(f)$ has no singularities.
        \end{corollary}

        \begin{proof}
        Assume $\rpts = \emptyset$ and suppose $\mathcal{V}_{\C}(f)$ has a singularity at the point $p$. Then the intersection multiplicity of $p$ at $f$ and $\frac{\partial f}{\partial t}$ is at least $2$ \cite{fulton}. By Proposition \ref{distinct}, $\cpts$ consists of $n(n-1)$ distinct points, which gives a contradiction.
        \end{proof}

    Corollary \ref{nosing} implies we need only consider two cases in order to prove Theorem \ref{main}: the case when $\mathcal{V}_{\C}(f)$ is smooth (Section \ref{smooth}) and the case when $\mathcal{V}_{\C}(f)$ has at least one real singularity (Section \ref{singular}). In either case, we consider a restriction of the following map.
   Define the linear map
    \begin{align}
    \phimap{}: \C[t,u,v] \rightarrow \C[t,u,v] \text{ where }
        h(t,u,v) \mapsto h(t,\omega u,\omega^{-1}v)
    \end{align}
for $\omega = e^{\frac{2\pi i}{n}}$.
The eigenvalues of $\phimap{}$ are $1,\omega,\dots, \omega^{n-1}$. Denote the corresponding eigenspaces of $\C[t,u,v]$ by $
            \phieig{1}, \phieig{\omega}, \dots, \phieig{\omega^{n-1}}$.
 Let 
     \begin{equation}\label{restriction}
         \phimap{k} :=     
                 \left.\phimap{}\right|_{\C[t,u,v]_k}
     \end{equation}
     be the restriction of $\phimap{}$ to $\C[t,u,v]_k$ and denote the corresponding eigenspaces of $\C[t,u,v]_k$ by 
 \begin{equation}
 \Ek{1}, \Ek{\omega}, \dots, \Ek{\omega^{n-1}}\text{.}
 \end{equation}
    Dixon's idea was to recover $M$, the determinantal representation of $f$, by first constructing $\adj(M)$. Plaumann and Vinzant altered this construction to produce Hermitian determinantal representations. Our desired representation is Hermitian with added structure, so in Section $\ref{smooth}$ we further modify the Hermitian construction to reflect that special structure. In particular, we require the entries of $\adj(M)$ to lie in the eigenspaces of $\phimap{n-1}$. 
In the next lemma we determine the dimension of these eigenspaces.

\begin{lemma}\label{dimEigen}
For all $\ell = 0, \ldots, n-1$,
    $\dim \Big(\El\Big) =  
            \begin{cases}{}
           \frac{n+1}{2}\text{,} & \text{ $n$ odd} \vspace{1mm}\\
     \frac{n}{2}\text{\hspace{.5mm},} & \text{ $n$ even and $\ell$ even} \vspace{1mm}\\
           \frac{n}{2}+1\text{,} & \text{ $n$ even and $\ell$ odd} \end{cases}\text{.}$
\end{lemma}

\begin{proof} Computing the dimension of eigenspace $\El$ is equivalent to counting the number of monomials $t^iu^jv^k \in \C[t,u,v]_{n-1}$ so that $\phimap{n-1}(t^iu^jv^k) = \omega^{\ell}t^iu^jv^k$. On the other hand, applying $\phimap{n-1}$ to this monomial gives $\phimap(t^iu^jv^k) = \omega^{j-k}t^iu^jv^k$. Therefore, we must compute the cardinality of the set  $\{(j,k) \mid j-k\equiv \ell \text{ mod } n, 0 \leq j, k, j+k \leq n-1\}\text{.}
    $
  Let $n$ be odd.  Therefore, to prove $\dim(\El)=\frac{n+1}{2}$, we must really show that for some fixed $\ell$ where $0 \leq \ell \leq n-1$, we have 
        $\# \{(j,k) \mid j-k\equiv \ell \text{ mod } n, 0 \leq j, k, j+k \leq n-1\} =
            \frac{n+1}{2}\text{.}
    $
    Let $\ell$ be even. Then if $j-k \geq 0$, we have $j-k = \ell$. Since $0 \leq j+k \leq n-1$, this implies $0 \leq 2k + \ell \leq n-1$, so $-\frac{\ell}{2} \leq k \leq \frac{n-1-\ell}{2}$. We also have $ 0 \leq k$ and $\ell$ is even, so this means $0 \leq k \leq \frac{n-1-\ell}{2}$. Therefore, there are $\frac{n-1-\ell}{2} - 0 + 1 = \frac{n+1-\ell}{2}$ possibilites when $j-k \geq 0$.
    If $j-k < 0$, then $j-k=\ell-n$ which implies $k-j = n-\ell$ and $k+j = n-\ell+2j$. Since $0 \leq j+k \leq n-1$, we have $0 \leq n-\ell+2j \leq n-1$, so $\frac{\ell-n}{2} \leq j \leq \frac{\ell - 1}{2}$. We also have $0 \leq j$ and $\ell-1$ is odd, so these inequalities imply $0 \leq j \leq \frac{\ell-2}{2}$. Therefore, there are $\frac{\ell-2}{2} + 1 = \frac{\ell}{2}$ possibilities when $j-k <0$. In total, we have counted $\frac{n+1}{2}$ pairs $(j,k)$ when $\ell$ is even.
    \\Now let $\ell$ be odd. Then if $j-k \geq 0$, we have $j-k = \ell$. Since $0 \leq j+k \leq n-1$, this implies $0 \leq 2k + \ell \leq n-1$, so $-\frac{\ell}{2} \leq k \leq \frac{n-1-\ell}{2}$. We also have $ 0 \leq k$ and $\ell$ is odd, so this means $0 \leq k \leq \frac{n-2-\ell}{2}$. Therefore, there are $\frac{n-2-\ell}{2} - 0 + 1 = \frac{n-\ell}{2}$ possibilites when $j-k \geq 0$.
    If $j-k < 0$, then $j-k=\ell-n$ which implies $k-j = n-\ell$ and $k+j = n-\ell+2j$. Since $0 \leq j+k \leq n-1$, we have $0 \leq n-\ell+2j \leq n-1$, so $\frac{\ell-n}{2} \leq j \leq \frac{\ell - 1}{2}$. We also have $0 \leq j$ and $\ell$ is even, so these inequalities imply $0 \leq j \leq \frac{\ell-1}{2}$. Therefore, there are $\frac{\ell-1}{2} + 1 = \frac{\ell+1}{2}$ possibilities when $j-k <0$. In total, we have counted $\frac{n+1}{2}$ pairs $(j,k)$ when $\ell$ is odd.
    Thus, we have shown this set has cardinality $\frac{n+1}{2}$ whether $\ell$ is even or odd. A similar counting argument follows for the case when $n$ is even.

\end{proof}

  Let 
  \begin{equation}\label{I}
  \I \text{ and } \It
  \end{equation}
  denote the space of all degree $n-1$ forms vanishing on the points $S$ in ($\ref{S}$) and $\tilde{S}$ in (\ref{Stilde}) respectively. The next proposition shows there must always exist some eigenvector $h \in \El$ for every $\ell$ which vanishes on points of $S$.


\begin{proposition} \label{dimVL}
 For all $\ell=0, \ldots, n-1$, $\dim \big(\El\cap \I\big) \geq 1$.
 \end{proposition}
 
 \begin{proof}
  First we show $\I \cap \El = \It \cap \El$.
  Let $h \in \I$. This means $h$ vanishes on all point of $S$, and since $\tilde{S} \subset S$, $h$ vanishes on all points of $\tilde{S}$. Therefore, $h \in \It$ and we have $\I \cap \El \subseteq \It \cap \El$.
  For the other direction, suppose $h\in \It \cap \El$, we need to show that $h(t,u,v)=0$, for all $[t:u:v]\in S$.
  Every $[t:u:v]\in~S$ belongs to an orbit under the action of rotation, meaning that there exists $[\tilde{t}:\tilde{u}:\tilde{v}] \in \tilde{S}$ with $ [t:u:v]=[\tilde{t}:\omega^{\mu}\tilde{u}:\omega^{-\mu}\tilde{v}]$ for some integer $\mu$. Then we have 
        $h(t,u,v)=h(\tilde{t},\omega^{\mu}\tilde{u},\omega^{-\mu}\tilde{v})
              = \omega^{\mu \ell} h(\tilde{t},\tilde{u},\tilde{v})=0  $, so $\It \cap \El \subseteq I(S)_{n-1} \cap \El$.
 Therefore, $\I \cap \El = \It \cap \El$. More importantly, these sets have the same dimension.
 When $n$ is odd, $\dim(\El)=\frac{n+1}{2}$ from Lemma \ref{dimEigen} and $|\tilde{S}| = \frac{n-1}{2}$. 
 Then 
 \begin{align*}
     \dim(\I \cap \El) &= \dim(\It \cap \El) \\
     &\geq \dim(\El) - |\tilde{S}| \\
     &= \frac{n+1}{2} - \frac{n-1}{2} = 1\text{.}
     \end{align*}
 When $n$ is even and $\ell$ is odd, $\dim(\El)=\frac{n}{2}+1$ from Lemma \ref{dimEigen} and $|\tilde{S}| = \frac{n}{2}$. This implies
 \begin{align*}
     \dim(\I \cap \El) \geq \dim(\El) - |\tilde{S}| = \frac{n}{2}+1 - \frac{n}{2} = 1\text{.}
     \end{align*}
 Now consider the case when $n$ and $\ell$ are both even. We have $\dim(\El)=\frac{n}{2}$ from Lemma \ref{dimEigen}, but the monomials $t^iu^jv^k$ so that $j-k=\ell$ and $i+j+k = n-1$ also satisfy $i \geq 1$ since $j-k$ is even, but $n-1$ is odd. This means all of the monomials in each $V(\omega^{\ell})$ in this case must have a factor of $t$. So, we need not consider the points with $t=0$. Therefore, $|\tilde{S}\backslash\{[0:u:v]\}| = \frac{n}{2}-1$ and
 \begin{align*}
     \dim(\I \cap \El) \geq \dim(\El) - |\tilde{S}| = \frac{n}{2} - \left(\frac{n}{2}-1\right) = 1\text{.}
     \end{align*}
 In every case we have shown $\dim(\I \cap \El) \geq 1$ and this completes the proof.
 \end{proof}

Another stipulation for the Hermitian construction is that the $2 \times 2$ minors of $\adj(M)$ must lie in the ideal $\langle f \rangle$. To ensure this is possible alongside the aforementioned eigenspace requirement, we use a fact due to Max Noether. 
This fact is developed mainly in the language of divisors. For information on divisors, see \cite{fulton}. The next lemma not only allows us to write the $2 \times 2$ minors as elements of $\langle f \rangle$, but also to choose each entry of $\adj(M)$ in an appropriate eigenspace of $\phimap{n-1}$.
    \begin{lemma} \label{symmofg}
       Suppose $f \in \phieig{1}$, $g \in \phieig{1}$ and $h \in \phieig{\omega^{\ell}}$ are homogeneous with $\mathcal{V}_{\C}(f)$ smooth where $\deg(h) > \deg(f), \deg(g)$  and $g$ and $h$ have no irreducible components in common with $f$. If $\mathcal{V}(f,g)$ consists of distinct points and $\mathcal{V}_{\C}(f,g) \subseteq \mathcal{V}(f,h)$, then there exists homogeneous $\hat{a}, \hat{b} \in \phieig{\omega^{\ell}}$ so that $h=\hat{a}f+\hat{b}g$ where $\deg\big(\hat{a}\big) = \deg(h)-\deg(f)$ and $\deg\big(\hat{b}\big) = \deg(h)-\deg(g)$. Additionally, if $f,g$, and $h$ are real, then $\hat{a}$ and $\hat{b}$ can be chosen real.
    \end{lemma}

\begin{proof}
     For reality of $\hat{a}$ and $\hat{b}$, see Theorem~$4.1$ of \cite{PlaumannVinzant}. By Max Noether's fundamental theorem~\cite{fulton}, there exists homogeneous $a, b \in \C[t,u,v]$ so that $h=af+bg$ where $\deg(a)=\deg(h)-\deg(f)$ and $\deg(b)=\deg(h)-\deg(g)$  since $\mathcal{V}(f,g) \subseteq \mathcal{V}(f,h)$ and $\mathcal{V}(f,g)$ consists of distinct points. Next, $h \in \phieig{\omega^{\ell}}$ implies $h = \omega^{-\ell} \phimap{}(h)$. Then 
        $$
        h = \dfrac{1}{n}\sum \limits_{i=0}^{n-1} \omega^{-\ell i}\phimap{}^i(h) = f\cdot\left(\dfrac{1}{n}\sum \limits_{i=0}^{n-1} \omega^{-\ell i}\phimap{}^i(a)\right) + g\cdot\left(\dfrac{1}{n}\sum \limits_{i=0}^{n-1} \omega^{-\ell i}\phimap{}^i(b)\right)\text{.}$$
        Let $\hat{b} = \dfrac{1}{n}\sum \limits_{i=0}^{n-1} \omega^{-\ell i}\phimap{}^i(b)$. Now we show that $\phimap{}(\hat{b}) = \hat{b}$. Applying the map, we have
        \begin{align*}
            \phimap{}\big(\hat{b}\big) = \phimap{}\left(\dfrac{1}{n}\sum \limits_{i=0}^{n-1} \omega^{-\ell i}\phimap{}^i(b)\right)
                &= \dfrac{1}{n}\sum \limits_{i=0}^{n-1} \omega^{-\ell i}\phimap{}^{i+1}(b) \\
                    & = \dfrac{1}{n}\sum \limits_{i=1}^{n} \omega^{-\ell (i-1)}\phimap{}^i(b) \\
                    & = \omega^\ell \left(\dfrac{1}{n} \sum \limits_{i=1}^{n} \omega^{-\ell i} \phimap{}^i(b)\right) \\
                    & = \omega^\ell \left(\dfrac{1}{n} \sum \limits_{i=1}^{n-1} \left( \omega^{-\ell i} \phimap{}^i(b)\right) + \omega^{-\ell n}\phimap{}^n(b) \right) \\
                    & = \omega^\ell \left(\dfrac{1}{n} \sum \limits_{i=1}^{n-1} \left( \omega^{-\ell i} \phimap{}^i(b)\right) + \omega^{0}\phimap{}^0(b) \right) \\
                    & = \omega^{\ell} \hat{b} \text{.}
        \end{align*}
        This means $\phimap{}\big(\hat{b}\big)=\omega^\ell \hat{b}$ and $\hat{b} \in \phieig{\omega^{\ell}}_{\deg(b)}$. The polynomial $\hat{a}$ is chosen in a similar fashion.
\end{proof}

\section{The Smooth Case} \label{smooth}




Now, with proper choices, we can recover a determinantal representation and its associated cyclic weighted shift matrix for a given curve invariant under $C_n$. Below we list the steps of Dixon's construction for Hermitian representations,  italicize our manipulations, and prove the main theorem for smooth curves with facts from Section \ref{smooth}.

\vspace{3mm}
\begin{construction} \rm Let $f \in \R[t,x,y]^{C_n}_n$ be hyperbolic with respect to $(1,0,0)$ with $f(1,0,0)=1$.
        \begin{enumerate}
            \item Write $f \in \R[t,u,v]_n^{C_n}$ as in (\ref{fInvz}) and let $g_{11}= \frac{\partial f}{\partial t}$ be of degree $n-1$.
            \item Split the $n(n-1)$ points of $\xi\left(\mathcal{V}_{\mathbb{C}}(f, g_{11})\right)$ into two conjugate sets of points $S \cup \overline{S}$ \emph{with $\frac{n-1}{2}$ sets of orbits of points in $S$ when $n$ is odd and $\frac{n}{2}$ sets of orbits of points in $S$ when $n$ is even.}
            \item Extend $g_{11}$ to a linearly independent set $\{g_{11}, g_{12}, \dots, g_{1n}\}$ of forms in $\mathbb{C}[t,u,v]_{n-1}$ vanishing on all points of $S$ \emph{with $g_{1j} \in \E{\omega^{1-j}}$ for all $j \in [n]$. }
            \item For $1 < i \leq j$, let $g_{ij}$ be a form for which  $g_{11} g_{ij} - \overline{g_{1i}} g_{1j}$ lies in the ideal generated by $f$, \emph{with $g_{ij} \in \E{\omega^{i-j}}$} and $g_{ii} \in \R[t,u,v]$.
              \item   For $i<j$, set $g_{ji} = \overline{g_{ij}}$ and define $G=(g_{ij})$  to be the resulting $n\times n$ matrix.
  \item   Define $M$ to be the matrix of linear forms obtained by dividing each entry of $\adj({G})$ by $f^{n-2}$. 
  \item Normalize $M$ so all diagonal entries are monic in $t$.
        \end{enumerate}
\end{construction}

    \begin{proof}[Proof of Theorem \ref{main}(a) (Smooth case)]
            Assume $\mathcal{V}_{\C}(f)$ is smooth. The goal is to show there exists cyclic weighted shift matrix $A \in \C^{n \times n}$ such that
            \begin{equation}\label{desired}
                f(t,u,v) =f_A(t,u,v) = \det\left( \begin{matrix}
t      & \frac{a_1}{2}v     & 0      & \cdots & 0  & \frac{\overline{a_n}}{2}u   \\
\frac{\overline{a_1}}{2}u       & t       & \ddots    & 0      & \cdots  & 0       \\
0 & \ddots  & \ddots & \ddots & \ddots  & \vdots  \\
\vdots & 0  & \ddots & \ddots & \ddots  & 0       \\
0      & \vdots       & \ddots & \ddots & \ddots  & \frac{a_{n-1}}{2}v \\
\frac{a_n}{2}v  & 0       & \cdots & 0 & \frac{\overline{a_{n-1}}}{2}u  & t      \\  
\end{matrix} \right)\text{.}
            \end{equation}
            We will construct a matrix $G$ of forms of degree $n-1$ and recover the desired representation by taking the adjugate. Let $g_{11} = \frac{\partial f}{\partial t}$. Split $\xi\left(\mathcal{V}_{\C}(f,g_{11})\right)$ into $ S \cup \overline{S}$ so $S$ consists of the appropriate number of orbits of points for odd and even $n$ as in (\ref{S}). All of these points are distinct by Proposition \ref{distinct}. For Step 3, Proposition \ref{dimEigen} allows us to choose $g_{1j} \in \E{\omega^{1-j}}$ which vanishes on all points of $S$ for each $j=1, \ldots, n$. Now let $g_{j1} = \overline{g_{1j}}$. These entries vanish on all points in $\overline{S}$ and $\mathcal{V}(f,g_{11}) \subseteq \mathcal{V}(f, g_{1j}g_{j1})$. By Lemma \ref{symmofg}, we can choose $g_{ij} \in \E{\omega^{i-j}}$ for $1 < i < j$ so that $g_{11}g_{ij}-\overline{g_{1i}}g_{1j} =af$ for some homogeneous $a \in \mathbb{C}[t,u,v]$ to complete Step 4. Let $g_{ij} = \overline{g_{ji}}$ for $i > j$ and let $G = (g_{ij})$ be the $n \times n$ matrix of forms of degree $n-1$. 
Next consider the restriction $\phimap{1}$ of $\phimap{}$ to  $\mathbb{C}[t,u,v]_1$. The eigenvalues of this restriction are $1, \omega$, and $\omega^{n-1}$ with associated eigenspaces $\F{1}$, $\F{\omega}$, and $\F{\omega^{n-1}}$. Since each $2 \times 2$ minor of $G$ lies in the ideal $\langle f \rangle$, each entry of $\adj(G)$ will be divisible by $f^{n-2}$ by Theorem~$4.6$ of \cite{PlaumannVinzant} and Step 6 is valid. Let $M = \frac{\adj(G)}{f^{n-2}}$.
The entries in $G$ have degree $n-1$, so the entries of $\adj(G)$ have degree $(n-1)^2$. Then $f^{n-2}$ has degree $n(n-2)$, so entries of $M$ are linear in $t, u,$ and $v$. 
Applying the map $\phimap{}$ to the $ij$-th entry of $M$, we have
            \begin{align*}
            \phimap{}(M)_{ij} &= f^{2-n} \phimap{}(\adj(G))_{ij}\\
            &= f^{2-n} \adj({\phimap{}}(G))_{ij}\\
            &= f^{2-n} \adj(\Omega G\Omega^{\ast})_{ij}\\
            &= f^{2-n} (\adj(\Omega^{\ast})\adj(G)\adj(\Omega))_{ij}\\
            &= (\Omega M\Omega^{\ast})_{ij}\\
            &= \overline{\omega^{j-1}}\omega^{i-1}M_{ij} \\
            &=\omega^{i-j}M_{ij}\text{.}
            \end{align*}       
			Therefore, $M_{ij} \in \F{\omega^{i-j}}$ for each $i,j$. This implies $M_{ij}=0$ if $|i-j| \neq 0, 1$. Now consider the entries of $M$ so that $i-j=0$ or $1$. These are the entries in the main and upper diagonals of $M$ as well as the $M_{1n}$  entry. For $M_{ij}$ such that $j=i+1$ and $i=1, j=n$, we have $M_{ij}\in \F{\omega^{n-1}}$ so these are multiples of $v$. Since $M$ is Hermitian, this implies $M_{ji} \in \F{\omega}$, so these entries are multiples of $u$. 
 Also, $M_{ii} \in \F{1}$, so the diagonal elements must be multiples of $t$. More explicitly, $M_{ii} = c_{i}t$ for some scalars $c_i$. To normalize as in Step 7, replace $M$ by $DMD$ where $D=\text{diag}\left(\frac{1}{\sqrt{c_{1}}},\ldots,\frac{1}{\sqrt{c_{n}}}\right)$ so that the coefficients of $t$ are $1$. Thus, the matrix $M$ can be reduced to the form (\ref{desired}).
	\end{proof}

\begin{example}[Quintic]
    \rm We will compute a determinantal representation of $$f(t,u,v)=t^5-\frac{25}{2}t^3uv+\frac{135}{4}t(uv)^2-3\sqrt{3}\left(1+\sqrt{2}\right)\left(\frac{u^5+v^5}{2}\right)+3\sqrt{3}\left(1-\sqrt{2}\right)\left(\frac{u^5-v^5}{2i}\right)$$ and identify the associated cyclic weighted shift matrix. Let $g_{11} = \frac{\partial f}{\partial t} = 5t^4 -\frac{75}{2}t^2uv+\frac{135}{4}(uv)^2$ and compute the points of $\mathcal{V}_{\C}(f,g)$. Split the points into $S\cup\overline{S}$ so $$\tilde{S} = \left\{ [1:-0.6426 + 0.4343 i: -1.0213 - 0.6902 i], [1:-0.4195 - 0.0142 i: -0.3689 + 0.0125 i] \right\}$$ is the set of orbit representatives of $S$ as in (\ref{Stilde}). Choose $g_{12} = -t^3v + 3tuv^2 - \frac{3\sqrt{3}}{2}\left(1-i\right)\left(\sqrt{2}+i\right)u^4$ which lies in $\Equintic{\omega^{-1}}$ and vanishes on $S$. Now for $j \leq 5$, choose $g_{1j} \in \Equintic{\omega^{1-j}}$ so $g_{1j}$ vanishes on the points of $S$ and set $g_{j1}=\overline{g_{1j}}$. Write $$g_{12}g_{21} = \left(-t^3+3tuv\right)f + \left(t^4-7t^2uv+6(uv)^2\right)g_{11} \in \langle f,g_{11} \rangle$$ and let $g_{22} = t^4-7t^2uv+6(uv)^2 \in \Equintic{1}$. For every other $i\leq j$, write $g_{1i}g_{j1} = af + bg_{11}$ for $a,b \in \C[t,u,v]$ where $b \in \Equintic{\omega^{i-j}}$ and set $g_{ij}=b, g_{ji}=\overline{b}$. Denote $G$ the matrix with $g_{ij}$ entries, take the adjugate of $G$ to get a matrix with entries in $\C[t,u,v]_{16}$, and divide each entry by $f^3$. One of the resulting representations is $10000 \cdot f(t,u,v) = \det(M)$ where
 $$M = \left( \begin{matrix}
2t      & 2\sqrt{5} v     & 0      & 0      & -4i\sqrt{5} u      \vspace{1mm}\\
2\sqrt{5} u      & 10 t       & \frac{15+15i}{\sqrt{2}} v    & 0      &  0      \vspace{1mm} \\
0 & \frac{15-15i}{\sqrt{2}} u & 5 t & 5\sqrt{3} v & 0 \vspace{1mm}\\
0 & 0 & 5\sqrt{3} u & 10 t & 5(\sqrt{2}+2i) v \vspace{1mm}\\
4i\sqrt{5} v & 0 & 0 & 5(\sqrt{2}-2i) u & 10 t
\end{matrix} \right) \text{.}$$
Finally, let $D = \text{diag}\left(\frac{1}{\sqrt{2}},\frac{1}{\sqrt{10}},\frac{1}{\sqrt{5}},\frac{1}{\sqrt{10}},\frac{1}{\sqrt{10}}\right)$, and normalize $M$ so that 
    $$f(t,u,v) = \det(DMD) = \det\left(tI_5 + uA^{\ast}+vA \right)$$ where $A = S\left(2, 3+3i, \sqrt{6}, \sqrt{2}+2i, -4i\right)$ is the associated cyclic weighted shift matrix.
\end{example}

\section{The Singular Case} \label{singular}

We have shown this construction holds when $\mathcal{V}_{\C}(f)$ is smooth, but it still remains valid if $\mathcal{V}_{\C}(f)$ is singular. In particular, the case of real singularities may be solved by reducing to a univariate argument. The next lemma is essential in this reduction. We believe  this result has been shown elsewhere, but include the proof here for completeness.
        \begin{lemma} \label{realroots}
        Let $p \in \mathbb{R}[t]$ and $a<b$ for $a, b  \in \mathbb{R}$. If $p(t)+a$ and $p(t)+b$ have all real roots, then $p(t)+c$ has all real, distinct roots for $c \in (a,b)$.
    \end{lemma}
    
    \begin{proof}
    Suppose $p(t)+a$ and $p(t)+b$ have all real roots. This implies $p'(t)$ must have $n-1$ real roots $r_1,\ldots, r_{n-1}$.
    Suppose $n$ is even. Then we have $\lim \limits_{t\to - \infty} p(t)+a = \infty$. This means that there must be a local minimum at $r_1$ due to the shape of the graph of $p(t)+a$. We also know $p(r_1)+a \leq 0$ or else $p(t)+a$ has imaginary roots. Similarly, there must be a local maximum at $r_2$ and $p(r_2)+a\geq 0$. Continuing in this fashion, $p(r_k)+a \leq 0$ is a local minimum for all odd $k$ and $p(r_k)+a \geq 0$ is a local maximum for all even $k$ where $k = 1,\ldots n-1$. The same argument holds for $p(t)+b$, so we have $p(r_k)+a < p(r_k)+b \leq 0$ for all odd $k$ and $p(r_k)+b > p(r_k)+a \geq 0$ for all even $k$ where $k = 1,\ldots n-1$ because $a < b$. More importantly, we have $c \in (a,b)$, so $p(r_k)+c < 0$ for all odd $k$ and $p(r_k)+c > 0$ for all even $k$.
	We have $\lim \limits_{t\to - \infty} p(t)+c = \infty$ and $p(r_1)+c < 0$. By Intermediate Value Theorem, there exists $s_1 \in (-\infty,r_1)$ such that $p(s_1)+c=0$. Similarly, since $p(r_1)+c < 0$ and $p(r_2)+c > 0$, there exists $s_2 \in (r_1,r_2)$ such that $p(s_2)+c=0$. Continue in this fashion. We have $p(r_{n-1})+c < 0$ and $\lim \limits_{t\to \infty} p(t)+c = \infty$. Then there exists $s_n \in (r_{n-1},\infty)$ such that $p(s_n)+c=0$. We have $p(s_i)+c=0$ for $i=1,\ldots,n$, so $p(t)+c$ has $n$ real roots.
	\\Now suppose $n$ is odd. Since $\lim \limits_{t\to - \infty} p(t)+c = - \infty$, we have $p(r_k)+a \geq 0$ is a local maximum for odd $k$ and $p(r_k)+a \leq 0$ is a local minimum for even $k$ where $k=1,\ldots,n-1$. A similar argument follows in the same way as the even case. Therefore, $p(t)+c$ has all real roots for $c \in (a,b)$.
    The roots of $p(t)+c$ are $s_1, s_2, \ldots, s_n$ such that $s_1 < r_1 < s_2 < r_2 < \ldots < r_{n-1} < s_n$ and, more importantly, are distinct.
    \end{proof}
    
The hyperbolicity of $f \in \R[t,x,y]_{n}$ is equivalent to real rootedness of two univariate polynomials and any real singularities of $\mathcal{V}(f)$ are related to repeated roots of these polynomials.

    \begin{proposition} \label{allrealroots}
        The polynomial $f$ is hyperbolic if and only if $t^n + \sum_{r=1}^{\floor*{\frac{n}{2}}} c_r t^{n-2r} \pm \sqrt{c_0^2 + \tilde{c_0}^2}$ have all real roots.
    \end{proposition}

    \begin{proof}
        By definition, $f$ is hyperbolic with respect to $(1,0,0)$ if and only if $f(t,\cos(\theta), \sin(\theta))$ has all real roots for all $\theta \in [0,2\pi)$. Then
     \begin{align*}
     f(t,\cos(\theta), \sin(\theta)) &= t^n + \sum_{r=1}^{\floor*{\frac{n}{2}}} c_r t^{n-2r} + c_0 \cos(n\theta)+\tilde{c_0}\sin(n\theta) \\
     &= t^n + \sum_{r=1}^{\floor*{\frac{n}{2}}} c_r t^{n-2r} + \sqrt{c_0^2+\tilde{c_0}^2} \left(\cos(\alpha)\cos(n\theta)+\sin(\alpha)\sin(n\theta)\right)\\
     &= t^n + \sum_{r=1}^{\floor*{\frac{n}{2}}} c_r t^{n-2r} + \sqrt{c_0^2+\tilde{c_0}^2}\cos(\alpha-n\theta)
     \end{align*}
    where $\frac{c_0}{\sqrt{c_0^2+\tilde{c_0}^2}} = \cos(\alpha)$ and $\frac{\tilde{c_0}}{\sqrt{c_0^2+\tilde{c_0}^2}} = \sin(\alpha)$ for some $\alpha \in [0, 2\pi)$.  By Lemma \ref{realroots} and since $-1 \leq \cos(\alpha-n\theta)\leq 1$, it is enough to check if $t^n + \sum_{r=1}^{\floor{\frac{n}{2}}} c_r t^{n-2r} \pm \sqrt{c_0^2+\tilde{c_0}^2}$ each have all \vspace{1mm}real roots to show $f(t,\cos(\theta), \sin(\theta))$ has all real roots for every $\theta \in [0,2\pi)$.
    \end{proof}

    \begin{lemma} \label{multipleroots}
       If $f$ is hyperbolic where $\mathcal{V}(f)$ has a real singularity, then at least one of $t^n+\sum_{r=1}^{\floor{\frac{n}{2}}} c_r t^{n-2r} \pm \sqrt{c_0^2+\tilde{c_0}^2}$ has a repeated root.
    \end{lemma}
    
    \begin{proof}
        Suppose $t^n + \sum_{r=1}^{\floor{\frac{n}{2}}} c_r t^{n-2r} \pm \sqrt{c_0^2+\tilde{c_0}^2}$ have all distinct real roots. By Lemma \ref{realroots}, this implies $f\left(t,\cos(\theta),\sin(\theta)\right)$ has all distinct real roots for every $\theta \in [0, 2\pi)$. Therefore $f\left(t,\cos(\theta),\sin(\theta)\right)$ and $\frac{\partial f}{\partial t}\left(t,\cos(\theta),\sin(\theta)\right)$ have no common roots in $t$. This holds for every $\theta \in [0, 2\pi)$, so $f\left(t,x,y\right)$ and $\frac{\partial f}{\partial t}\left(t,x,y\right)$ have no real intersection points. In other words, $\mathcal{V}_{\mathbb{R}}\left(f,\frac{\partial f}{\partial t}\right)=\emptyset$, but $$\mathcal{V}_{\mathbb{R}}\left(f,\frac{\partial f}{\partial t},\frac{\partial f}{\partial x},\frac{\partial f}{\partial y}\right) \subseteq \mathcal{V}_{\mathbb{R}}\left(f,\frac{\partial f}{\partial t}\right)\text{.}$$ Therefore, $f$ has no real singularities. 
    \end{proof}
    Equivalently, if neither of these univariate polynomials have a repeated root, then $\mathcal{V}(f)$ has no real singularities. Now we may prove the remaining singular case of Theorem \ref{main}(a) and again use the equivalent hyperbolicity condition in (\ref{hyp}) to our advantage.

    \begin{proof}[Proof of Theorem \ref{main}(a) (Singular case)]
         We dealt with the case $\mathcal{V}_{\C}(f)$ smooth in Section \ref{smooth}. Suppose $\mathcal{V}_{\C}(f)$ has a singularity and at least one of $c_0, \tilde{c_0}$ are nonzero. Write 
         $$f = t^n + \sum_{r=1}^{\floor*{\frac{n}{2}}} c_r t^{n-2r}(uv)^r + c_0 \left(\frac{u^n+v^n}{2}\right)+\tilde{c_0}\left(\frac{u^n-v^n}{2i}\right)$$ as in (\ref{fInvz}). Let $p(t)=t^n + \sum_{r=1}^{\floor*{\frac{n}{2}}} c_r t^{n-2r}$ and $s = \sqrt{c_0^2+\tilde{c_0}^2}$. Then by Lemma \ref{allrealroots} and Lemma \ref{multipleroots}, each of $p(t) \pm s$ has all real roots and at least one has a repeated root. Either $s \neq 0$ or $s = 0$.  If $s \neq 0$ (at least one of $c_0, \tilde{c_0}$ are nonzero), define $$f_\varepsilon(t,u,v)= t^n + \sum_{r=1}^{\floor*{\frac{n}{2}}} c_r t^{n-2r} (uv)^{r} +(c_0-\text{sign}(c_0)\varepsilon) \left(\frac{u^n+v^n}{2}\right)+(\tilde{c_0}-\text{sign}(\tilde{c_0})\varepsilon)\left(\frac{u^n-v^n}{2i}\right)\text{,}$$ so $ \lim\limits_{\varepsilon \to 0} f_\varepsilon =f$ and  by Lemma \ref{realroots}, $f_\varepsilon(t,e^{i\theta},e^{-i\theta})=p(t)+s_\varepsilon$ has all real distinct roots for $$s_{\varepsilon}=\sqrt{(c_0-\text{sign}(c_0)\varepsilon)^2+(\tilde{c_0}-\text{sign}(\tilde{c_0})\varepsilon)^2}\text{.}$$ 
         If $s=0$, perturb the nonzero coefficients of the univariate polynomial $p(t)$ to get $p_{\varepsilon}(t)$ with all real distinct roots and $ \lim\limits_{\varepsilon \to 0} p_\varepsilon =p$. Let $p_{\varepsilon}(t)^{\text{hom}}$ be the homogenization of $p_{\varepsilon}(t)$ with respect to $uv$ and define $$f_{\varepsilon}(t,u,v) = p_{\varepsilon}(t)^{\text{hom}} + \varepsilon\left( \frac{u^n+v^n}{2}\right)\text{.}$$ Then by Proposition \ref{realroots}, $f_{\varepsilon}(t,e^{i\theta},e^{-i\theta}) = p_{\varepsilon}(t) + \varepsilon \cos(n \theta)$ has all real distinct roots for every $\theta \in [0, 2\pi)$.
         In either case, this means $\mathcal{V}_{\C}\left(f_\varepsilon \right)$ is smooth and for every $\varepsilon >0$ there exists cyclic weighted shift matrix $W_\varepsilon \in \mathbb{C}^{n \times n}$ such that $f_\varepsilon = f_{W_{\varepsilon}}$ by Lemma \ref{multipleroots}. Now $f_{\varepsilon}(t,-1,-1)$ and $f_{\varepsilon}(t,-i,i)$ are the characteristic polynomials of $\Re(W_\varepsilon)$ and $\Im(W_\varepsilon)$ and converge to the roots of $f(t,-1,-1)$ or $f(t,-i,i)$ respectively. Therefore, the eigenvalues of $\Re(W_\varepsilon)$ and $\Im(W_\varepsilon)$ are bounded, which bounds the sequences $(\Re(W_\varepsilon))_\varepsilon$ and $(\Im(W_\varepsilon))_\varepsilon$. Then
         $$ (\Re(W_\varepsilon))_\varepsilon + i(\Im(W_\varepsilon))_\varepsilon= (\Re(W_\varepsilon) + i\Im(W_\varepsilon))_\varepsilon= (W_\varepsilon)_\varepsilon$$ which is also bounded. By passing to a convergent subsequence, this means $\lim \limits_{\varepsilon \to 0}(W_\varepsilon)_\varepsilon = W$ and $$f=\det\left(\lim \limits_{\varepsilon \to 0}\left(tI_n + \frac{u}{2}W_\varepsilon^{\ast}+\frac{v}{2}W_\varepsilon\right)\right) = \det\left(tI_n + \frac{u}{2}W^{\ast}+\frac{v}{2}W\right)\text{.}$$
    \end{proof}

    This proof provides an analogue to a result of Nuij who showed the space of smooth (homogeneous) hyperbolic polynomials is dense in the space of all (homogeneous) hyperbolic polynomials \cite{nuij}. More specifically, since every $f \in \R[t,x,y]^{C_n}_n$ hyperbolic with respect to $(1,0,0)$ where $\mathcal{V}_{\R}(f) \neq \emptyset$ is the limit of $f' \in \R[t,x,y]^{C_n}_n$ hyperbolic with respect to $(1,0,0)$ where $\mathcal{V}_{\C}(f')$ smooth, we have the following corollary.
    \begin{corollary} The space of smooth, hyperbolic forms of degree $n$ invariant under $C_n$ is dense in the space of all hyperbolic forms of degree $n$ invariant under $C_n$.
    \end{corollary}

\section{Dihedral Invariance}\label{dihedral}

Recall that in addition to rotation about the axis of $t$, the dihedral group of order $n$ is also generated by a reflection that maps the point $[t:u:v]$ to $[t:v:u]$. Since $\C[t,u,v]^{D_n}_n \subset \C[t,u,v]^{C_n}_n$ and the invariant $\frac{u^n-v^n}{2i}$ of $C_n$ does not remain the same under the action of reflection, polynomials with dihedral invariance have the form 
    \begin{equation} \label{dihedralf}
f(t,u,v) = t^n + \sum_{r=1}^{\floor*{\frac{n}{2}}} c_r t^{n-2r} (uv)^{r} +c_0 \left(\frac{u^n+v^n}{2}\right)\text{.}
    \end{equation} 
    Now we can prove Theorem \ref{main}(b) which gives a positive answer to Chien and Nakazato's main question in \cite{Chien2013}.

    \begin{proof}[Proof of Theorem \ref{main}(b).]
            Note $f \in \R[t,u,v]^{D_n} \subseteq \R[t,u,v]^{C_n}$, so by Theorem \ref{main}(a) there exists cyclic weighted shift matrix $A \in \C^{n \times n}$ where $f(t,u,v) = f_A(t,u,v) = \det\left(tI_n + \frac{u}{2}A^{\ast} + \frac{v}{2}A \right)$. The polynomial $f$ has the form (\ref{dihedralf}), so $$t^n + 2^{n-1}c_0 = f(t,0,2) = f_A(t,0,2) = \det(tI + A) = t^n - a_1 a_2 \cdots a_n\text{,}$$ which implies $a_1 a_2 \cdots a_n \in \R$ since $c_0 \in \R$. Write $a_j = r_je^{i \alpha_j}$ with $r_j \in \R$ for each $j$, so $a_1 a_2 \cdots a_n \in \R$ and $\alpha_1 + \alpha_2 + \ldots + \alpha_n = 0$. Let $U = \text{diag}\left(e^{i\theta_1},e^{i\theta_2},\ldots, e^{i\theta_n}\right)$ for some $\theta_i \in [0, 2\pi)$. Since $UU^{\ast} = I$, we say $A$ is unitarily equivalent to $B=UAU^{\ast}$, so $f = f_A = f_B$. The matrix $B$ is the cylic weighted shift matrix $B=S\left(r_1e^{i (\theta_1-\theta_2+\alpha_1)},r_2e^{i (\theta_2-\theta_3+\alpha_2)}, \ldots, r_ne^{i (\theta_n-\theta_1+\alpha_n)}\right)$. Choose 
            \begin{align*}
                \theta_j &= -\alpha_j - \alpha_{j+1} - \ldots - \alpha_{n-1} \text{ for } 1 \leq j < n \\
                \theta_n &= -\alpha_1 - \alpha_2 - \ldots - \alpha_n = 0\text{.}
            \end{align*} This gives $\theta_{j} - \theta_{j+1} = -\alpha_j$ for $1 \leq j < n$ and $\theta_n - \theta_1 = -\alpha_n$, so $B = S(r_1, \ldots, r_n) \in \R^{n \times n}$.
    \end{proof}

\begin{example}[Quartic]
    \rm We will compute a determinantal representation of $$f(t,u,v)=t^4-26t^2uv+72(uv)^2-72\left(\frac{u^4+v^4}{2}\right)$$ and identify the associated cyclic weighted shift matrix. Let $g_{11} = \frac{\partial f}{\partial t} = 4t^3 -52tuv$ and compute the points of $\mathcal{V}_{\C}(f,g)$. Split the points into $S\cup\overline{S}$ so $$\tilde{S} = \left\{ \left[1:\frac{1+i}{\sqrt{39}}: \frac{\sqrt{3}}{2\sqrt{13}}(1-i)\right], \left[0:1:1\right] \right\}$$ is the set of orbit representatives of $S$ as in (\ref{Stilde}). Choose $g_{12} = -4t^2v-36u^3+36uv^2$ which lies in $\Equartic{\omega^{-1}}$ and vanishes on $S$. Now for $j \leq 4$, choose $g_{1j} \in \Equartic{\omega^{1-j}}$ so $g_{1j}$ vanishes on the points of $S$ and set $g_{j1}=\overline{g_{1j}}$. Write $$g_{12}g_{21} = \left(-4t^2+36uv\right)f + \left(t^3-18tuv\right)g_{11} \in \langle f,g_{11} \rangle$$ and let $g_{22} = t^3-18tuv \in \Equartic{1}$. For every other $i\leq j$, write $g_{1i}g_{j1} = af + bg_{11}$ for $a,b \in \C[t,u,v]$ where $b \in \Equartic{\omega^{i-j}}$ and set $g_{ij}=b, g_{ji}=\overline{b}$. Denote $G$ the matrix with $g_{ij}$ entries, take the adjugate of $G$ to get a matrix with entries in $\C[t,u,v]_{9}$, and divide each entry by $f^2$. One of the resulting representations is $1728 \cdot f(t,u,v) = \det(M)$ where
 $$M = \left( \begin{matrix}
3t      &  12v     & 0           & (9+9\sqrt{3}i) u      \vspace{1mm}\\
12 u      & 12 t       & \big(6\sqrt{2}(1-i)+2\sqrt{6}(1+i)\big) v    & 0         \vspace{1mm} \\
0 & \big(6\sqrt{2}(1+i)+2\sqrt{6}(1-i)\big) u & 4 t & 6\sqrt{6}(1+i) v  \vspace{1mm}\\
\big(9-9\sqrt{3}i\big)v & 0 & 6\sqrt{6}(1-i) u & 12 t \vspace{1mm}\\
\end{matrix} \right) \text{.}$$
Finally, let $D = \text{diag}\left(\frac{1}{\sqrt{3}},\frac{1}{\sqrt{12}},\frac{1}{2},\frac{1}{\sqrt{12}}\right)$, and normalize $M$ so that 
    $$f(t,u,v) = \det(DMD) = \det\left(tI_4 + uA^{\ast}+vA \right)$$ where 
    $$ A = S\left(1,\sqrt{2}\big(1-i\big)+\sqrt{6}\big(1+i\big),3\sqrt{2}\big(1+i\big),3+3\sqrt{3}i\right)=S(4, 4e^{\frac{\pi i}{12}}, 6e^{\frac{\pi i}{4}},6e^{\frac{-\pi i}{3}})$$ is the associated cyclic weighted shift matrix. Let $U = \text{diag}\big(e^{\frac{-\pi i}{3}},e^{\frac{-\pi i}{3}},e^{\frac{-\pi i}{4}},1\big)$. Then $A$ is unitarily equivalent to $B = UAU^{\ast} = S(4,4,6,6)$ with all real entries. 
\end{example}

\section{Conclusion} 
By properly modifying a construction of Dixon \cite{Dixon}, we have shown that each polynomial $f$ that is hyperbolic with respect to $(1,0,0)$ with cyclic invariance always has a determinantal representation via some cyclic weighted shift matrix. If the polynomial has dihedral invariance, we have shown one can find a determinantal representation associated to a cyclic weighted shift matrix with real entries. One can find such representations using computer algebra systems, but the construction quickly gets more difficult as $n$ grows larger.
 In practice, computing these constructions is not so easy. The difficulty lies in the computation of intersection points of $f$ and its partial derivative. We would like to implement an algorithm to output cyclic weighted shift determinantal representations given a polynomial with the appropriate properties. Additionally, we'd like to consider the action of other finite groups on polynomials of varying degree and associated determinantal representations.

\section*{Acknowledgments}
We would like to thank Daniel Plaumann and Cynthia Vinzant for many helpful discussions. The second author received support from the National Science Foundation (DMS-1620014). 

\bibliography{mybibfile}

\end{document}